\documentclass[a4paper,11pt]{article}

\usepackage[utf8]{inputenc} 
\usepackage[T1]{fontenc}
\usepackage{graphicx}
\usepackage{multirow}
\usepackage{amsmath,amssymb,amsfonts}
\usepackage{amsthm}
\usepackage[title]{appendix}
\usepackage{xcolor}
\usepackage{textcomp}
\usepackage{manyfoot}
\usepackage{booktabs}
\usepackage{enumitem}
\usepackage{authblk}
\usepackage{hyperref} 
\theoremstyle{plain}
\newtheorem{theorem}{Theorem}
\newtheorem{proposition}[theorem]{Proposition}
\newtheorem{lemma}[theorem]{Lemma}
\newtheorem{corollary}[theorem]{Corollary}

\theoremstyle{remark}
\newtheorem{remark}{Remark}

\theoremstyle{definition}
\newtheorem{definition}{Definition}

\newcommand{\mB}{\mathcal{B}}
\newcommand{\mXlambda}{\mathcal{X}(\lambda)}

\title{On the Complexity of Linear Programs with Parametric Constraint Matrices}

\author[1]{Laurie Boveroux}
\author[2]{Margarida Carvalho}
\author[3]{Andrea Lodi}
\author[1]{Quentin Louveaux}

\affil[1]{Department of Computer Science and Electrical Engineering, University of Liège, Liège, Belgium}
\affil[2]{Department of Computer Science and Operations Research, Université de Montréal, Montréal, Québec, Canada}
\affil[3]{Cornell Tech and Technion, New York, NY, USA}

\date{}

\begin{document}

\maketitle

\date{}

\begin{abstract}
While sensitivity analysis for linear programs with parametric right-hand side vectors is well established, cases involving parametric constraint matrices remain less explored.
To address this gap, we investigate linear programs where the constraint matrix depends linearly on a scalar parameter.
Our goal is to compute the minimum and maximum optimal values of these parametric linear programs within a fixed parameter range. This corresponds to solving two bilevel programs.
We show that both decision versions of these optimization problems are NP-complete by proving both NP-hardness and membership in NP.
We additionally identify three sufficient conditions on the parametric matrix under which the bilevel programs are no longer NP-hard.
We introduce a reformulation for both problems, together with a Farkas-based certificate for the lower-level feasibility condition, and compare the execution time for solving them on standard benchmark instances.
The experiments highlight a practical dichotomy. Although both bilevel optimizations are NP-hard, computing the parameter value corresponding to the minimum of the linear program is fast, whereas computing the value corresponding to the maximum is not.
We relate this asymmetry to the parametric equality constraints, whose free dual variables leave the relaxation of the maximum reformulation unbounded.
\end{abstract}



\maketitle


\section{Introduction}
Parametric linear programming studies how the optimal value of a linear program evolves under perturbations.
While classical settings consider perturbations of the objective or the right-hand side, much less is known when the constraint matrix itself is affected.
In this work, we study linear programs parameterized by a single scalar parameter $\lambda$ varying over a given interval $\Lambda \subset \mathbb{R}$, where the parameter linearly modifies some coefficients of the constraint matrix.
For a given realization of $\lambda \in \Lambda$, we denote the corresponding linear program by $\mathcal{P}(\lambda)$ and formalize it as
\begin{equation*}
\begin{aligned}
    f(\lambda) \;=\; \min_{\mathbf{x} \in \mathbb{R}^n} \quad & \mathbf{c}^\top \mathbf{x} \\
    \text{s.t.} \quad & A \mathbf{x}+ \lambda D\mathbf{x} \le \mathbf{b}, \\
    & \mathbf{x} \ge \mathbf{0}
\end{aligned}
\tag{$\mathcal{P}(\lambda)$} \label{prob:parametric_lp}
\end{equation*}
where $\mathbf{x} \in \mathbb{R}^n$ is the decision vector, $A \in \mathbb{Q}^{m \times n}$ is the nominal constraint matrix, $D \in \mathbb{Q}^{m \times n}$ is the perturbation matrix associated with $\lambda$ and $\mathbf{b} \in \mathbb{Q}^m$ is the right-hand side vector.
The function $f(\lambda)$ is called the \emph{value function} of the parametric linear program.
For instance, this setting appears when a single physical or economic parameter influences multiple constraints simultaneously. 
In time-indexed models of production or energy systems, parameters such as efficiency or resource availability appear repeatedly across constraints at each time period.
Therefore, varying such a parameter affects multiple coefficients in the constraint matrix.

In such models, the objective typically represents total operating or investment costs, or a performance metric.
Understanding how the objective value evolves with respect to the parameter $\lambda$ is interesting to assess the impact of uncertainty on the objective value and to identify the worst-case and best-case scenario.
In particular, determining the global extrema of the associated value function over a given interval,
\[
\min_{\lambda \in [\lambda_1, \lambda_2]} f(\lambda)
\quad \text{and} \quad
\max_{\lambda \in [\lambda_1, \lambda_2]} f(\lambda),
\]
where $\lambda_1, \lambda_2 \in \mathbb{Q}$, provides exact bounds on the achievable cost or performance as the parameter varies.
These extrema lead to bilevel formulations, where $\lambda$ is selected at the upper level and the linear program is solved at the lower level.
However, the function $f(\lambda)$ may be neither convex nor concave, making these problems difficult to solve.
An illustrative example is shown in Figure~\ref{fig:nonconvex_example}, which depicts the value function of the following parametric linear program over $\lambda \in [-1, 1]$:
\begin{equation*}
\label{eq:nonconvex_example}
\begin{aligned}
f(\lambda) = \min_{x_0, x_1} \quad & -x_1 \\
\text{s.t.} \quad & (-1 + \lambda)\, x_0 + x_1 \le 1, \\
& 2\lambda\, x_0 + (1 - \lambda)\, x_1 \le 1, \\
& (1 - \lambda)\, x_0 - (1 + \lambda)\, x_1 \le 1, \\
& x_0, x_1 \ge 0.
\end{aligned}
\end{equation*}
This function has two local minima, around $\lambda \approx -0.35$ and $\lambda \approx 0.4$, separated by a local maximum at $\lambda = 0$.
Hence, $f$ is neither convex nor concave, and not even quasi-convex or quasi-concave.

\begin{figure}[ht]
\centering
\includegraphics[width=0.9\textwidth]{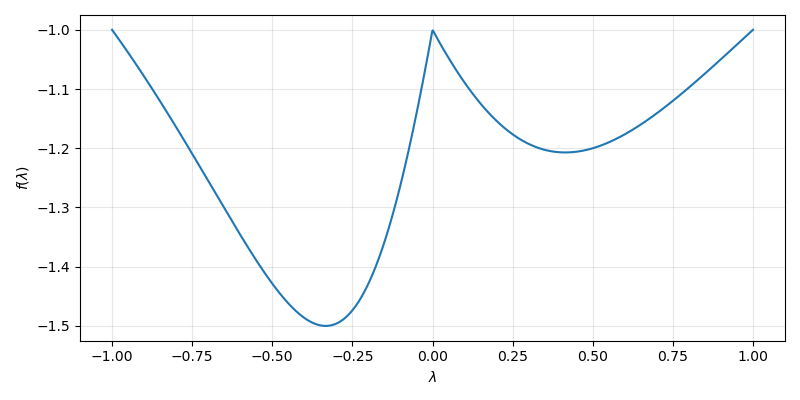}
\caption{Value function $f(\lambda)$ of the illustrative example over $\Lambda=[-1,1]$: $f$ is neither convex nor concave.}\label{fig:nonconvex_example}
\end{figure}

\subsection{Related work}
Three streams of research are related to our work: sensitivity analysis, robust optimization, and bilevel optimization.

Parametric and sensitivity analysis in linear programming study how optimal solutions and objective values evolve when problem data are perturbed.
Studies mainly focus on variations in the objective coefficients or in the right-hand side.
Perturbations on the constraint matrix have received less attention.
Among the works that explicitly consider parameter-dependent constraint matrices, Zuidwijk \cite{zuidwijk2005linear} studies linear programs whose constraint matrix depends on a scalar parameter and derives local expressions for the evolution of the optimal value.
Zuidwijk's analysis relies on the characterization of regions in which a given basis remains optimal, leading to a piecewise description of the value function.
Similarly, Khalilpour and Karimi \cite{khalilpour2014parametric} consider uncertainty on the left-hand side and propose a parametric solution approach based on basis-dependent computations.

More recently, Derval et al.~\cite{miftari2024sensitivity} propose bounding procedures to estimate the value function over a parameter interval, yielding computable upper and lower bounds.
Derval et al.~\cite{derval2025efficient} study the evaluation of optimal solutions and objective values for a finite set of parameter values and develop warm-start strategies to accelerate repeated solves.
Both works consider the same parametric model as the present paper, but focus primarily on its computational treatment.
Their contributions are complementary to ours, which studies the computational complexity of optimizing the value function over an entire parameter interval and develops exact single-level reformulations for computing its minimum and maximum.

A related line of research is robust optimization \cite{ben2009robust}, whose goal is to find a solution that optimizes the objective function under the worst-case 
realization of uncertain data while remaining feasible across all potential data realizations.
This framework is also heavily motivated by parameter uncertainty within the constraint matrix coefficients, e.g., see~\cite{soyster1973convex,ben1999robust,bertsimas2004price}. In contrast to this focus on identifying a robust decision, our objective is to characterize the best- and worst-case values reached by the optimal value function of the linear program when the scalar parameter varies over a given interval.

Our work is also connected to bilevel optimization.
Indeed, optimizing the value function with respect to the parameter yields a bilevel formulation, where the upper-level variable is the scalar parameter $\lambda$ and the lower level corresponds to a linear program.
From a complexity perspective, the decision version of bilevel linear optimization problems is NP-complete; Jeroslow~\cite{jeroslow1985polynomial} established NP-hardness and Buchheim~\cite{buchheim2023bilevel} showed inclusion in NP.
This NP-hardness persists under other restrictive settings, such as min-max variants of bilevel linear optimization problems~\cite{hansen1992new}.
Moreover, bilevel optimization problems remain NP-hard even when the upper level involves a single decision variable~\cite{sugishita2025complexity}, or when considering the pessimistic variant with a fixed number of lower-level variables~\cite{ketkov2025note}.
Despite this, our problem does not fall directly into the class of bilevel linear optimization problems.
Indeed, the dependence on $\lambda$ appears in the constraint matrix of the lower-level problem and breaks the linearity.
Thus, existing complexity results for bilevel linear optimization cannot be directly transferred to our setting.
Related complexity analyses have been conducted for bilevel problems involving bilinear terms, such as bilevel quadratic programs, toll-pricing and combinatorial pricing~\cite{ketkov2025note,roch2005approximation,bui2024asymmetry,grune2025complexity}.
In these formulations, the bilinearities arise in the objective function from interactions between upper- and lower-level decision variables, whereas in our setting they involve a single scalar parameter $\lambda$ that affects the feasible region.

\subsection{Our contributions}
In this paper, we study parametric linear programs in which the constraint matrix depends linearly on a scalar parameter. As formalized in Section~\ref{sec:statement}, we  investigate the problem of computing the parameter values that minimize and maximize the optimal value of the resulting linear program over a given interval.
In Section~\ref{sec:complexity}, we prove that both decision versions of the optimization problems are NP-complete by establishing polynomial-time verifiable certificates together with polynomial reductions from known NP-complete problems.
We then identify three sufficient conditions on the parametric matrix under which both bilevel problems become polynomially solvable.
In Section~\ref{sec:reformulations}, we derive single-level reformulations for both problems. To this end, dualization of the lower-level is standard, but here it is valid only over parameter values where the lower-level is feasible and bounded. We therefore pair the strong duality reformulation of the maximum problem with a Farkas-based certificate. 
Finally, in Section~\ref{sec:experiments}, we confront the theoretical complexity picture with practice on benchmark instances proposed by Derval et al. \cite{miftari2024sensitivity}. Although both problems are NP-hard, the minimum reformulation can be solved efficiently, whereas computing the maximum remains significantly more challenging. Section~\ref{sec:conclusions} concludes the paper.

\section{Problem statement}\label{sec:statement}
We consider the parametric linear optimization problem \ref{prob:parametric_lp}.
We are interested in computing its global extrema over the interval $\Lambda = [\lambda_1, \lambda_2]$.
This leads to the following optimization problems:
\begin{equation}
\min_{\lambda \in [\lambda_1,\lambda_2]} \;\; \min_{\mathbf x \ge 0} \; \left\{ \mathbf c^\top \mathbf x \;\middle|\; A\mathbf x + \lambda D \mathbf x \leq \mathbf b \right\} \label{min--min}
\end{equation}
and
\begin{equation}
\max_{\lambda \in [\lambda_1,\lambda_2]} \;\; \min_{\mathbf x \ge 0} \; \left\{ \mathbf c^\top \mathbf x \;\middle|\; A\mathbf x + \lambda D \mathbf x \leq \mathbf b \right\}. \label{max--min}
\end{equation}
These formulations can be interpreted as bilevel optimization problems in which the upper level selects the parameter $\lambda$, while the lower level solves the linear program \ref{prob:parametric_lp}.

\begin{remark}
    \label{remark:bilevel_inf_unb}
    We note that for $\lambda \in [\lambda_1,\lambda_2]$ such that the lower-level is unbounded, our min--min problem~\eqref{min--min} is unbounded. Moreover, for $\lambda \in [\lambda_1,\lambda_2]$ such that the lower-level is infeasible, our max--min problem~\eqref{max--min} is unbounded. However, if the lower-level was formalized instead as the following constraint within our bilevel programs:
    $$\mathbf{x} \in \arg \min_{\mathbf x' \ge 0} \; \left\{ \mathbf c^\top \mathbf x' \;\middle|\; A\mathbf x' + \lambda D \mathbf x' \leq \mathbf b \right\},$$
    a $\lambda \in [\lambda_1,\lambda_2]$ such that (i) the lower-level is unbounded would result in min--min problem~\eqref{min--min} being infeasible, and (ii) the lower-level is infeasible would result in max--min problem~\eqref{max--min} being infeasible. We adopt the formulations~\eqref{min--min} and~\eqref{max--min} because our goal is to bound the value function over $\Lambda$: they return $f(\lambda) = -\infty$ and $f(\lambda) = +\infty$ as the correct extrema, whereas the argmin-constrained formulation would instead declare the bilevel problem infeasible. This detail is relevant for the NP membership and algorithmic approaches considered in this paper.
\end{remark}

\subsection{Auxiliary reformulations}
\label{subsec:aux}

We first observe that the min--min problem~\eqref{min--min} can be rewritten as a single-level problem.
Since the upper and lower levels have the same objective function, the bilevel formulation reduces to a joint minimization over $(\mathbf x,\lambda)$:
\begin{equation*}
\min_{\mathbf{x}, \, \lambda} \quad \mathbf{c}^\top \mathbf{x} \quad \text{s.t.} \quad (A + \lambda D)\mathbf{x} \le \mathbf{b}, \quad \mathbf{x} \ge \mathbf{0}, \quad \lambda_1 \le \lambda \le \lambda_2 \tag{$\mathcal{P}_{\min}$}\label{prob:p_min}
\end{equation*}
In contrast, the max--min problem cannot be reduced to a joint optimization problem,
since the lower-level variable must be optimal for the value of $\lambda$ chosen at the upper level.
To handle this structure, we rely on a dual reformulation of the lower-level problem. Its dual is
\[
\max_{\mathbf y \le \mathbf 0} \left\{
\mathbf b^\top \mathbf y \;\middle|\;
(A+\lambda D)^\top \mathbf y \le \mathbf c
\right\}.
\]
If we assume that the linear program $\mathcal P(\lambda)$ is feasible for all $\lambda$ and has a finite optimal value, by strong duality, we therefore have
\[
\min_{\mathbf x \ge \mathbf 0} \left\{
\mathbf c^\top \mathbf x \;\middle|\;
(A+\lambda D)\mathbf x \le \mathbf b
\right\}
=
\max_{\mathbf y \le \mathbf 0} \left\{
\mathbf b^\top \mathbf y \;\middle|\;
(A+\lambda D)^\top \mathbf y \le \mathbf c
\right\}.
\]
These two auxiliary reformulations of the lower-level (i.e., primal and dual representations) will be used throughout the paper.
The joint reformulation of the min--min problem is used in the complexity analysis and in the computational experiments.
The dual reformulation of the lower-level problem is used in the complexity analysis of the max--min problem and in the derivation of its single-level formulation.

\section{Complexity}\label{sec:complexity}
Given a rational threshold $K$, the decision versions of min--min~\eqref{min--min} and max--min~\eqref{max--min} are, respectively
\begin{equation}
   \exists \lambda\in\Lambda \text{ such that } f(\lambda)\le K,
    \tag{D-min}
     \label{D-min}
\end{equation}
and
\begin{equation}
    \exists \lambda\in\Lambda \text{ such that } f(\lambda)\ge K.
    \tag{D-max}
    \label{D-max}
\end{equation}

In this section, we show the following main theorem:
\begin{theorem}\label{theorem:main}
The decision problems~\ref{D-min} and~\ref{D-max} 
are NP-complete.
\end{theorem}
In Sections~\ref{nphardmin} and~\ref{nphardmax}, we prove NP-hardness by showing that min--min~\eqref{min--min} and max--min~\eqref{max--min} contain as a special case the linear multiplicative programming problem. In Section~\ref{npminmax}, we show membership of~\ref{D-min} and~\ref{D-max} in NP.
We then, in Section~\ref{sec:tractable}, identify restricted assumptions on $D$ under which the problems are polynomially solvable.

\subsection{NP-hardness of the min--min problem} \label{nphardmin}
As shown in Section~\ref{subsec:aux}, the bilevel formulation~\eqref{min--min} is equivalent to the reformulation~\eqref{prob:p_min}.
Therefore, it is sufficient to prove NP-hardness of this single-level formulation.

\begin{theorem}
The min--min problem~\eqref{prob:p_min} is NP-hard.
\label{thmmin}
\end{theorem}

\begin{proof}
We show that every member of Matsui's NP-hard family of linear multiplicative programs~\cite{matsui1996np} can be rewritten in polynomial time as an instance of~\eqref{prob:p_min} with the same optimal value. This family of linear multiplicative programs is therefore a special case of~\eqref{prob:p_min}, from which NP-hardness follows. Matsui's family is
\begin{equation*}
\min \quad z_1 z_2 \tag{$P_1(M)$}\label{prob:p1_m}
\end{equation*}
with decision variables
\[
\mathbf u =(u_1,\dots,u_n)\in\mathbb{R}^n,\qquad
\mathbf v=(v_{ij})_{1\le i,j\le n}\in\mathbb{R}^{n\times n},
\qquad u^0,v^0,z_1,z_2\in\mathbb{R},
\]
subject to
\begin{align}
0\le u_i\le 1 &\qquad (i=1,\dots,n), \label{eq:u-bounds}\\
0\le v_{ij}\le 1 &\qquad (i,j=1,\dots,n), \label{eq:v-bounds}\\
v_{ij}\le u_i,\qquad v_{ij}\le u_j,\qquad v_{ij}\ge u_i+u_j-1
&\qquad (i\neq j), \label{eq:mc}\\
v_{ii}=u_i &\qquad (i=1,\dots,n), \label{eq:diag-rel}\\
M\mathbf u=\mathbf 1,& \label{eq:Matsui-partition}\\
u^0=\sum_{i=1}^n p^i u_i,& \label{eq:u0-def}\\
v^0=\sum_{i=1}^n\sum_{j=1}^n p^{i+j} v_{ij},& \label{eq:v0-def}\\
z_1=v^0-p+2p^{4n}+2p^{2n}u^0,& \label{eq:z1-def}\\
z_2=v^0-p+2p^{4n}-2p^{2n}u^0,& \label{eq:z2-def}
\end{align}
where $M$ is an $m\times n$ binary matrix with $n>m$ and $p=n^{n^4}$.

\noindent
We now write~\eqref{prob:p1_m} as an instance of~\eqref{prob:p_min} with decision vector
\[
\mathbf x=(\,s,\ u_1,\dots,u_n,\ v_{11},\dots,v_{nn},\ u^0,\ v^0,\ z_2,\ t\,),
\]
and parameter $\lambda$ representing $z_1$. The objective is $\mathbf c^\top \mathbf x=t$.

\noindent
All linear constraints \eqref{eq:u-bounds}–\eqref{eq:v0-def} and \eqref{eq:z2-def} are included in $A \mathbf x\le \mathbf b$, together with $s=1$ enforced by $s\le 1$, $-s\le -1$. The auxiliary variable $s$ is needed because $\lambda$ enters~\eqref{prob:p_min} only through the products $\lambda D\mathbf x$; the product $\lambda s$ therefore plays the role of $\lambda$ itself.

\noindent
The relation \eqref{eq:z1-def} is encoded as
\[
\lambda s=v^0-p+2p^{4n}+2p^{2n}u^0,
\]
i.e.,
\begin{align*}
\lambda s-v^0-2p^{2n}u^0 &\le -p+2p^{4n},\\
-\lambda s+v^0+2p^{2n}u^0 &\le p-2p^{4n}.
\end{align*}
The epigraph constraint
\[
-t+\lambda z_2\le 0
\]
enforces $t\ge \lambda z_2$.

\noindent
From the bounds on $u$ and $v$, we obtain
\[
-p+2p^{4n}\le \lambda \le n^2p^{2n}-p+2p^{4n}+2np^{3n}.
\]
For $n\ge 5$, this yields $\lambda\in[p^{4n},5p^{4n}]$ and we set $\lambda_1=p^{4n}$, $\lambda_2=5p^{4n}$.
These bounds preserve all feasible values arising from $(P_1(M))$.
Moreover, $\lambda_1$ and $\lambda_2$ are polynomially encodable, since $\log p = O(n^4 \log n)$ \cite{matsui1996np}.

\noindent
Given any feasible solution of $(P_1(M))$, define $\mathbf x=(1,\mathbf u,\mathbf v,u^0,v^0,z_2,z_1z_2)$ and $\lambda=z_1$. Then, $(\mathbf x,\lambda)$ is feasible for $(\mathcal P_{\min})$ and $\mathbf c^\top \mathbf x=z_1z_2$, hence
\[
\operatorname{opt}(\mathcal P_{\min})\le \operatorname{opt}(P_1(M)).
\]

\noindent
Conversely, let $(\mathbf x,\lambda)$ be any feasible solution for \eqref{prob:p_min}. Then, $s=1$ and the constraints imply
\[
\lambda=v^0-p+2p^{4n}+2p^{2n}u^0.
\]
Setting $z_1=\lambda$, we obtain a feasible solution of $(P_1(M))$ whose objective value satisfies $z_1z_2=\lambda z_2\le t=\mathbf c^\top\mathbf x$, by the epigraph constraint.
Since this holds for every feasible solution of~\eqref{prob:p_min},
\[
\operatorname{opt}(P_1(M))\le \operatorname{opt}(\mathcal P_{\min}).
\]

\noindent
Thus, $\operatorname{opt}(\mathcal P_{\min})=\operatorname{opt}(P_1(M))$ and NP-hardness follows.
\end{proof}

\subsection{NP-hardness of the max--min problem} \label{nphardmax}
In contrast with the min--min formulation, which reduces to a single-level optimization problem, a similar reduction is not possible at first glance 
for the max--min formulation, which remains bilevel. Thus, we expect the max--min problem, which involves a ``true'' bilevel structure, to be at least NP-hard.
The following theorem establishes this result.

\begin{theorem}\label{thmmax}
The max--min program~\eqref{max--min} is NP-hard.
\end{theorem}

\begin{proof}
We show that the specific min--min instance built in the proof of Theorem~\ref{thmmin} can in turn be rewritten in polynomial time as an instance of~\eqref{max--min} whose optimal value is its negative.
Composing the two constructions, every linear multiplicative program~\eqref{prob:p1_m} is rewritten as an instance of~\eqref{max--min}, from which NP-hardness follows.

\noindent
Consider the min--min instance built in the proof of Theorem~\ref{thmmin}, written as
\begin{equation*}
\min_{\lambda \in [\lambda_1,\lambda_2]} \;
\min_{\boldsymbol{\mu} \ge \mathbf 0} \left\{ \mathbf c_M^\top \boldsymbol{\mu} \;\middle|\; A_M \boldsymbol{\mu} + \lambda D_M \boldsymbol{\mu} \le \mathbf b_M \right\}.
\end{equation*}
Since the feasible region of~\eqref{prob:p1_m} is nonempty and bounded~\cite[Corollary~2.3]{matsui1996np}, this instance is feasible and its optimal value is finite. Note also that $\mathbf c_M \ge \mathbf 0$, since the objective of that instance is the single variable $t$.

\noindent
Consider now the max--min problem
\[
\max_{\lambda \in [\lambda_1,\lambda_2]} \;
\min_{\mathbf x \ge \mathbf 0} \left\{ \mathbf c^\top \mathbf x \;\middle|\; (A + \lambda D) \mathbf x \le \mathbf b \right\}.
\]
For a fixed value of $\lambda$, the inner minimization problem is a linear program.
Using the dual reformulation introduced in Section~\ref{subsec:aux} and the change of variables $\boldsymbol\mu = -\mathbf y \ge \mathbf 0$, its dual can be expressed as
\begin{equation*}
-\min_{\boldsymbol{\mu} \ge \mathbf 0} \left\{ \mathbf b^\top \boldsymbol{\mu} \;\middle|\; -A^\top \boldsymbol{\mu} + \lambda (-D)^\top \boldsymbol{\mu} \le \mathbf c \right\}.
\end{equation*}
We now choose the data $(A,D,\mathbf b,\mathbf c)$ so that this dual problem coincides with the min--min instance above. Setting
\begin{equation*}
A := -A_M^\top, \qquad D := -D_M^\top, \qquad \mathbf b := \mathbf c_M, \qquad \mathbf c := \mathbf b_M,
\end{equation*}
and keeping the same interval $[\lambda_1,\lambda_2]$, the dual of the inner problem becomes
\[
-\min_{\boldsymbol{\mu} \ge \mathbf 0} \left\{ \mathbf c_M^\top \boldsymbol{\mu} \;\middle|\; A_M \boldsymbol{\mu} + \lambda D_M \boldsymbol{\mu} \le \mathbf b_M \right\}. \qquad \textrm{($\boldsymbol\mu$-problem)}
\]

\noindent
Fix $\lambda \in [\lambda_1,\lambda_2]$ and consider the $\boldsymbol\mu$-problem obtained by fixing $\lambda$ in the min--min instance, whose feasible set is $\{\boldsymbol\mu \ge \mathbf 0 \mid A_M \boldsymbol\mu + \lambda D_M \boldsymbol\mu \le \mathbf b_M\}$. Whenever this set is nonempty, the $\boldsymbol\mu$-problem is bounded: its objective is the variable $t$, the epigraph constraint gives $t \ge \lambda z_2$, and $\lambda$ and $z_2$ are bounded, the latter because the construction of Theorem~\ref{thmmin} preserves the feasible points of~\eqref{prob:p1_m}, whose feasible region is bounded~\cite[Corollary~2.3]{matsui1996np}. Two cases arise. If the feasible set is nonempty, the $\boldsymbol\mu$-problem has a finite optimal value; the same then holds for its dual, which is the inner problem of the max--min instance, and strong duality yields
\[
\min_{\mathbf x \ge \mathbf 0} \left\{ \mathbf c^\top \mathbf x \;\middle|\; (A + \lambda D) \mathbf x \le \mathbf b \right\}
=
-\min_{\boldsymbol{\mu} \ge \mathbf 0} \left\{ \mathbf c_M^\top \boldsymbol{\mu} \;\middle|\; A_M \boldsymbol{\mu} + \lambda D_M \boldsymbol{\mu} \le \mathbf b_M \right\}.
\]
If instead the feasible set is empty, the $\boldsymbol\mu$-problem is infeasible and the right-hand side above equals $-\infty$, under the convention that a minimization over an empty feasible set equals $+\infty$. The left-hand side equals $-\infty$ as well: the point $\mathbf x = \mathbf 0$ is feasible for the inner problem since $\mathbf b = \mathbf c_M \ge \mathbf 0$, and an infeasible dual then forces the inner problem to be unbounded. The identity above therefore holds for every $\lambda \in [\lambda_1,\lambda_2]$.

\noindent
By substituting this expression into the max--min formulation, we have 
\begin{align*}
&\max_{\lambda \in [\lambda_1, \lambda_2]} \min_{\mathbf x \ge \mathbf 0} 
\left\{ \mathbf c^\top \mathbf x \;\middle|\; (A+\lambda D) \mathbf x \le \mathbf b \right\} \\
&\qquad = \max_{\lambda \in [\lambda_1, \lambda_2]} 
\left( -\min_{\boldsymbol{\mu} \ge \mathbf 0} 
\left\{ \mathbf c_M^\top \boldsymbol{\mu} \;\middle|\; A_M \boldsymbol{\mu} + \lambda D_M \boldsymbol{\mu} \le \mathbf b_M \right\} \right) \\
&\qquad = -\min_{\lambda \in [\lambda_1, \lambda_2]} \; \min_{\boldsymbol{\mu} \ge \mathbf 0} 
\left\{ \mathbf c_M^\top \boldsymbol{\mu} \;\middle|\; A_M \boldsymbol{\mu} + \lambda D_M \boldsymbol{\mu} \le \mathbf b_M \right\}.
\end{align*}
The constructed max--min instance therefore has optimal value $-\operatorname{opt}(P_1(M))$ by Theorem~\ref{thmmin}, and computing it is NP-hard.
\end{proof}

\subsection{Membership in NP} \label{npminmax}
We now show that the decision versions of Problems~\eqref{min--min} and \eqref{max--min} introduced at the beginning of this section belong to NP. 

\begin{theorem}
The decision problems \ref{D-min} and \ref{D-max} belong to NP.
\label{npinclusion}
\end{theorem}
Our proof builds on the certificate approach used by Buchheim~\cite{buchheim2023bilevel} to establish membership in NP for bilevel linear programming.
In that setting, an optimal basis of the lower-level problem serves as a certificate, and once the basis is fixed, verification reduces to checking the feasibility of a system of linear constraints.
Two additional difficulties arise in our setting.
First, because the lower-level constraint matrix depends on $\lambda$, fixing a basis yields rational conditions in $\lambda$, which become univariate polynomial constraints.
Second, as observed in Remark~\ref{remark:bilevel_inf_unb}, the lower-level problem may be unbounded or infeasible for some values of $\lambda$, in which case an
optimal basis certificate need not exist.
We therefore extend the strategy of~\cite{buchheim2023bilevel} by showing that the resulting univariate polynomial conditions can be verified in polynomial time and by using ray certificates to handle unboundedness and infeasibility.
\begin{proof}

We rewrite the lower-level linear program~\ref{prob:parametric_lp} in standard form by introducing slack variables $\mathbf{s} \in \mathbb{R}^m_{\ge 0}$. Setting $N = n+m$, $\mathbf{z} = (\mathbf{x}, \mathbf{s}) \in \mathbb{R}^N_{\ge 0}$, and $\mathbf{\bar c} = (\mathbf{c}, \mathbf{0}) \in \mathbb{R}^N$, the value function becomes
\[ 
f(\lambda)=\min_{\mathbf{z} \in \mathbb{R}^N} \left\{\mathbf{\bar c}^\top \mathbf{z}  \;\middle|\;  \mathbf{z} \in \mXlambda:= \left\{\mathbf{z} \ge \mathbf{0} \ \middle|\ \left[ A+ \lambda D \;\middle|\; I_m \right] \mathbf{z}= \mathbf{b}\right\}\right\}.
\]

\subsubsection{Proof for \ref{D-min}}
Problem \ref{D-min} is a YES-instance if and only if there exists $\lambda \in [\lambda_1, \lambda_2]$ such that $\mXlambda \neq \emptyset$ and either $f(\lambda) \le K$ or $f(\lambda) = -\infty$. 

The nondeterministic polynomial-time machine for \ref{D-min} operates as follows:
\begin{enumerate}[label=(\roman*)]
    \item Nondeterministically guess a basis $\mB \subseteq \{1, \ldots, N\}$ with $|\mB| = m$, and an indicator $j \in \bar{\mB} \cup \{\emptyset\}$, where $\bar{\mB} = \{1, \ldots, N\} \setminus \mB$.
    \item\label{step:optimalii}  If $j = \emptyset$, verify whether there exists $\lambda \in [\lambda_1, \lambda_2]$ such that $\mB$ yields a basic feasible solution satisfying $\mathbf{\bar c}^\top \mathbf{z}(\lambda) \le K$.
    \item\label{step:unboundediii}  If $j \in \bar{\mB}$, verify whether there exists $\lambda \in [\lambda_1, \lambda_2]$ such that $\mB$ is primal feasible and $j$ induces an unbounded recession ray certifying $f(\lambda) = -\infty$.
    \item Accept if the verification in either step (ii) or step (iii) succeeds.
\end{enumerate}

To perform the verification, let $M_{\mB}(\lambda)$ be the $m \times m$ submatrix of $[A+\lambda D \mid I_m]$ corresponding to $\mB$. By Cramer's rule, the basic variables $\mathbf{z}_{\mB}(\lambda) = M_{\mB}(\lambda)^{-1}\mathbf{b}$ and the components of the direction vector $\mathbf{d}_{\mB}^{(j)}(\lambda) = -M_{\mB}(\lambda)^{-1}\left[A+\lambda D \mid I_m\right]_j$ are the rational functions in $\lambda$
\[
z_{{\mB},i}(\lambda) = \frac{p_{\mB,i}(\lambda)}{q_{\mB}(\lambda)}, 
\qquad
d_{{\mB},i}^{(j)}(\lambda) = \frac{r_{\mB,i}^{(j)}(\lambda)}{q_{\mB}(\lambda)}, 
\qquad
q_{\mB}(\lambda) = \det(M_{\mB}(\lambda)),
\]
where $p_{\mB,i}(\lambda) = \det(M_{\mB,i}(\lambda))$ and $r_{\mB,i}^{(j)}(\lambda) = \det(M_{\mB,i}^{(j)}(\lambda))$, with $M_{\mB,i}(\lambda)$ and $M_{\mB,i}^{(j)}(\lambda)$ obtained by replacing the $i$-th column of $M_{\mB}(\lambda)$ with $\mathbf{b}$ and $-\left[A+\lambda D \mid I_m\right]_j$, respectively. By Lemma~\ref{lem:parametric-det} in Appendix~\ref{appendix:np}, the polynomials $p_{\mB,i}(\lambda)$, $r_{\mB,i}^{(j)}(\lambda)$, and $q_{\mB}(\lambda)$ have degree at most $m$ with polynomially bounded coefficients.

For $j = \emptyset$, step~\ref{step:optimalii} checks the feasibility over $\lambda \in [\lambda_1, \lambda_2]$ of
\begin{equation} \label{sys:dmin_bounded}
\mathbf{z}_{\mB}(\lambda) \ge \mathbf{0},  \qquad \mathbf{\bar c}_{\mB}^\top \mathbf{z}_{\mB}(\lambda) \le K, \qquad q_{\mB}(\lambda) \neq 0.
\end{equation}
For $j \in \bar{\mB}$, step~\ref{step:unboundediii} sets $d_j^{(j)}(\lambda) = 1$ and $d_k^{(j)}(\lambda) = 0$ for $k \in \bar{\mB} \setminus \{j\}$, and checks the feasibility over $\lambda \in [\lambda_1, \lambda_2]$ of
\begin{equation} \label{sys:dmin_unbounded_rigorous}
\mathbf{z}_{\mB}(\lambda) \ge \mathbf{0}, \qquad \mathbf{d}_{\mB}^{(j)}(\lambda) \ge \mathbf{0}, \qquad \mathbf{\bar c}^\top \mathbf{d}^{(j)}(\lambda) \le -1, \qquad q_{\mB}(\lambda) \neq 0.
\end{equation}

Write $K=\frac{u}{v}$, where $u\in\mathbb{Z}$ and $v\in\mathbb{Z}^+_0$.
Multiplying the objective inequality by $v$ and then the inequalities in \eqref{sys:dmin_bounded} and \eqref{sys:dmin_unbounded_rigorous} by $\left(q_{\mB}(\lambda)\right)^2$ clears all denominators without altering inequality signs. To enforce $q_{\mB}(\lambda) \neq 0$, we include the strict inequality $\left(q_{\mB}(\lambda)\right)^2 > 0$. In either step, this yields an equivalent system of univariate polynomial inequalities in $\lambda$ of degree at most $2m$ with coefficients of polynomial encoding length.
By Lemma~\ref{lem:rational-to-integer} in Appendix~\ref{appendix:np}, the resulting system can be transformed in polynomial time into an equivalent system of univariate polynomial constraints with integer coefficients. Therefore, Lemma~\ref{lem:univariate} in Appendix~\ref{appendix:np} applies and feasibility of this resulting system can be decided in polynomial time. This completes the proof that \ref{D-min} belongs to NP.

\subsubsection{Proof for \ref{D-max}}
By Remark~\ref{remark:bilevel_inf_unb}, if the lower-level problem~\ref{prob:parametric_lp} is infeasible, $f(\lambda) = +\infty \ge K$ holds trivially. Otherwise, for any lower-level dual feasible solution $\mathbf{y} \le \mathbf{0}$ satisfying $(A+\lambda D)^\top \mathbf{y} \le \mathbf{c}$, weak duality implies $f(\lambda) \ge \mathbf{b}^\top \mathbf{y}$. Hence, \ref{D-max} is a YES-instance if and only if there exists $\lambda \in [\lambda_1, \lambda_2]$ at which the dual program of~\ref{prob:parametric_lp} is feasible with optimal value at least $K$, or~\ref{prob:parametric_lp} is infeasible. Thus, the rest of the proof follows a logic similar to the one for \ref{D-min}, with step~\ref{step:optimalii} being adapted to focus on the optimality of the dual basis and with step~\ref{step:unboundediii}  being adapted to focus on the infeasibility of the lower-level.

We express the dual of $\mathcal{P}(\lambda)$ in standard form by introducing slack variables $\mathbf{v} \in \mathbb{R}^n_{\ge 0}$ and defining $\mathbf{w} = (-\mathbf{y}, \mathbf{v}) \in \mathbb{R}^N_{\ge 0}$, namely
\[
\max_{\mathbf{w} \ge \mathbf{0}} \left\{ -\bar{\mathbf{b}}^\top \mathbf{w} \;\middle|\; \left[-(A+\lambda D)^\top \;\middle|\; I_n \right] \mathbf{w} = \mathbf{c} \right\},
\]
where $\bar{\mathbf{b}} = (\mathbf{b}, \mathbf{0}) \in \mathbb{R}^N$. Let $M_D(\lambda) = [-(A+\lambda D)^\top \mid I_n] \in \mathbb{R}^{n \times N}$ be the dual constraint matrix. 

The nondeterministic polynomial-time machine for \ref{D-max} operates as follows:
\begin{enumerate}[label=(\roman*)]
    \item Nondeterministically guess a dual basis $\mB_y \subseteq \{1, \ldots, N\}$ with $|\mB_y| = n$, and an indicator $j \in \bar{\mB}_y \cup \{\emptyset\}$, where $\bar{\mB}_y = \{1, \ldots, N\} \setminus \mB_y$.
    \item If $j = \emptyset$, verify whether there exists $\lambda \in [\lambda_1, \lambda_2]$ such that $\mB_y$ yields a dual basic feasible solution satisfying $-\bar{\mathbf{b}}^\top \mathbf{w}(\lambda) \ge K$.
    \item If $j \in \bar{\mB}_y$, verify whether there exists $\lambda \in [\lambda_1, \lambda_2]$ such that $j$ induces a ray in the homogeneous dual system certifying primal infeasibility via Farkas' Lemma.
    \item Accept if the verification in either step (ii) or step (iii) succeeds.
\end{enumerate}

Let $M_{D,\mB_y}(\lambda)$ be the $n \times n$ submatrix of $M_D(\lambda)$ corresponding to $\mB_y$. By Cramer's rule, the dual basic variables $\mathbf{w}_{\mB_y}(\lambda) = M_{D,\mB_y}(\lambda)^{-1}\mathbf{c}$ and the basic components of the dual direction vector $\mathbf{d}_y^{(j)}(\lambda) = -M_{D,\mB_y}(\lambda)^{-1}\left[M_D(\lambda)\right]_j$ are rational functions in $\lambda$ with common denominator determinant $q_{\mB_y}(\lambda) = \det(M_{D,\mB_y}(\lambda))$.

For $j = \emptyset$, step (ii) checks the feasibility over $\lambda \in [\lambda_1, \lambda_2]$ of
\begin{equation} \label{sys:dmax_bounded}
\mathbf{w}_{\mB_y}(\lambda) \ge \mathbf{0}, \qquad \mathbf{w}_{\bar{\mB}_y}(\lambda) = \mathbf{0}, \qquad -\bar{\mathbf{b}}^\top \mathbf{w}(\lambda) \ge K, \qquad q_{\mB_y}(\lambda) \neq 0.
\end{equation}
For $j \in \bar{\mB}_y$, step (iii) sets $d_{y,j}^{(j)}(\lambda) = 1$ and $d_{y,k}^{(j)}(\lambda) = 0$ for $k \in \bar{\mB}_y \setminus \{j\}$, certifying primal infeasibility by checking the feasibility over $\lambda \in [\lambda_1, \lambda_2]$ of
\begin{equation} \label{sys:dmax_infeasible}
\mathbf{d}_{y,\mB_y}^{(j)}(\lambda) \ge \mathbf{0}, \qquad -\bar{\mathbf{b}}^\top \mathbf{d}_y^{(j)}(\lambda) \ge 1, \qquad q_{\mB_y}(\lambda) \neq 0.
\end{equation}

Write $K=\frac{u}{v}$, where $u\in\mathbb{Z}$ and $v\in\mathbb{Z}^+_0$.
Multiplying the objective inequality by $v$ and then the inequalities in \eqref{sys:dmax_bounded} and \eqref{sys:dmax_infeasible} by $\left(q_{\mB_y}(\lambda)\right)^2$ clears all denominators without altering inequality signs, and enforcing $\left(q_{\mB_y}(\lambda)\right)^2 > 0$ guarantees non-singularity. This converts both conditions into equivalent systems of univariate polynomial inequalities in $\lambda$ of degree at most $2n$ with coefficients of polynomial encoding length.
By Lemma~\ref{lem:rational-to-integer} in Appendix~\ref{appendix:np}, the resulting system can be transformed in polynomial time into an equivalent system of univariate polynomial constraints with integer coefficients. Therefore, Lemma~\ref{lem:univariate} in Appendix~\ref{appendix:np} applies and feasibility over $\left[\lambda_1, \lambda_2\right]$ is decidable in polynomial time. This completes the proof that \ref{D-max} belongs to NP.
\end{proof}

\begin{proof}[Proof of Theorem~\ref{theorem:main}]
The statement follows directly by combining the NP-hardness results, Theorems~\ref{thmmin} and~\ref{thmmax}, with the membership one in
Theorem~\ref{npinclusion}.
\end{proof}

\subsection{Tractable cases}\label{sec:tractable}
We now identify restricted assumptions on $D$ under which the problem of finding either the minimum or the maximum is no longer NP-hard. 

\subsubsection{Uniform sign in \texorpdfstring{$D$}{D}}
We consider the case where all coefficients of $D$ have the same sign.
In this setting, the value function $f(\lambda)$ is monotone on the interval $[\lambda_1,\lambda_2]$.
Then, the optimal values are attained at the endpoints of the interval.

\begin{proposition}\label{prop:uniform_sign}
Assume that all entries of $D$ have the same sign.
\begin{itemize}
    \item If $D_{ij} \geq 0 \ \forall i,j$, then $f(\lambda)$ is nondecreasing on $[\lambda_1,\lambda_2]$.
    
    \item If $D_{ij} \leq 0 \ \forall i,j$, then $f(\lambda)$ is nonincreasing on $[\lambda_1,\lambda_2]$.
\end{itemize}
\end{proposition}

\begin{proof}
Let $\lambda',\lambda'' \in [\lambda_1,\lambda_2]$ with $\lambda' \le \lambda''$.

\medskip
\noindent
Assume first that $D_{ij} \geq 0 \ \forall i,j$.
Since $\mathbf x \ge \mathbf 0$, we have $D\mathbf x \ge \mathbf 0$ and therefore
\[
\lambda' D\mathbf x \le \lambda'' D\mathbf x.
\]
Hence,
\[
A\mathbf x + \lambda' D\mathbf x
\le
A\mathbf x + \lambda'' D\mathbf x.
\]
It follows that every feasible solution for $\lambda''$ is also feasible for $\lambda'$, i.e.,
\[
F(\lambda'') \subseteq F(\lambda'),
\]
where $F(\lambda)$ denotes the feasible set of $\mathcal P(\lambda)$.
Therefore,
\[
f(\lambda') \le f(\lambda''),
\]
showing that $f$ is nondecreasing on $[\lambda_1,\lambda_2]$.

\medskip
\noindent
The case $D_{ij} \leq 0 \ \forall i,j$ follows similarly and yields that $f$ is nonincreasing on $[\lambda_1,\lambda_2]$.
\end{proof}

\begin{corollary}
Assume that all entries of $D$ have the same sign.
\begin{itemize}
    \item If $D_{ij} \geq 0 \ \forall i,j$, then
    \[
    \min_{\lambda \in [\lambda_1,\lambda_2]} f(\lambda)=f(\lambda_1),
    \qquad
    \max_{\lambda \in [\lambda_1,\lambda_2]} f(\lambda)=f(\lambda_2).
    \]

    \item If $D_{ij} \leq 0 \ \forall i,j$, then
    \[
    \min_{\lambda \in [\lambda_1,\lambda_2]} f(\lambda)=f(\lambda_2),
    \qquad
    \max_{\lambda \in [\lambda_1,\lambda_2]} f(\lambda)=f(\lambda_1).
    \]
\end{itemize}

Consequently, problems~\eqref{min--min} and~\eqref{max--min} are polynomially solvable,
since computing $f(\lambda)$ for a single value of $\lambda$ can be done in polynomial time.
\end{corollary}

\noindent
\textbf{Observation:} Having all coefficients of $D$ with the same sign cannot occur for models with parametric equality constraints, since such an equality is written as two inequalities whose coefficients have opposite signs to match the canonical model we consider in this paper.

\subsubsection{Exactly one non-zero row in \texorpdfstring{$D$}{D}}
We consider the case where $D$ has exactly one non-zero row.
This implies that only one constraint, denoted by $i_0$, is affected by the parameter $\lambda$.
The problem can then be written as
\begin{equation*}
\begin{aligned}
    \min_{\mathbf{x}, \, \lambda} \quad & \mathbf{c}^\top \mathbf{x} \\
    \text{s.t.} \quad & B \mathbf{x} \le \boldsymbol{\beta}, \\
    & \mathbf{a_{i_0}} \mathbf x + \lambda \mathbf{d_{i_0}} \mathbf x \leq b_{i_0}, \\
    & \mathbf{x} \ge \mathbf{0}, \\
    & \lambda \in [\lambda_1, \lambda_2]
\end{aligned}
\tag{$\mathcal{P}_{\min,\text{row}}$}\label{prob:p_min_row}
\end{equation*}
In this setting, we prove that the value function $f(\lambda)$ is quasi-concave on $[\lambda_1,\lambda_2]$.
As a consequence, the minimum over $\lambda$ is reached at one of the bounds $\lambda_1$ or $\lambda_2$.
The maximum can also be found in polynomial time, but requires a bisection procedure. 
\begin{definition}
A function $f$ is quasi-concave on an interval if for any $\lambda', \lambda''$ in the interval and any $\alpha \in [0,1]$,
\[
f(\alpha \lambda' + (1-\alpha)\lambda'') \geq \min\{f(\lambda'), f(\lambda'')\}.
\]
\end{definition}
\begin{proposition}\label{prop:one_row}
If $D$ has exactly one non-zero row, then the value function $f(\lambda)$ is quasi-concave on $[\lambda_1, \lambda_2]$.
\end{proposition}
\begin{proof}
Let $\lambda',\lambda''\in[\lambda_1,\lambda_2]$ and define
\[
\hat\lambda=\alpha\lambda'+(1-\alpha)\lambda'', \qquad \alpha\in[0,1].
\]
Let $\mathbf x^{\star}\geq \mathbf 0$ be an optimal solution at $\hat\lambda$. Then,
\[
B\mathbf x^{\star}\le \boldsymbol{\beta},\quad (\mathbf{a_{i_0}} + \hat{\lambda} \mathbf{d_{i_0}}) \mathbf x^{\star} \leq b_{i_0}.
\]
Suppose, by contradiction, that $\mathbf x^{\star}$ is infeasible for both $\lambda'$ and $\lambda''$. Then,
\[
(\mathbf{a_{i_0}} + \lambda' \mathbf{d_{i_0}}) \mathbf x^{\star} > b_{i_0}
\qquad\text{and}\qquad
(\mathbf{a_{i_0}} + \lambda'' \mathbf{d_{i_0}}) \mathbf x^{\star} > b_{i_0}.
\]
Multiplying the first inequality by $\alpha \geq 0$ and the second by $1-\alpha \geq 0$, and summing, leads to
\[
(\mathbf{a_{i_0}} + \hat{\lambda} \mathbf{d_{i_0}}) \mathbf x^{\star} > b_{i_0},
\]
which contradicts feasibility at $\hat\lambda$.

Hence, $\mathbf x^{\star}$ is feasible for at least one of the problems associated with $\lambda'$ or $\lambda''$.
If it is feasible for $\lambda'$, then by optimality of $f(\lambda')$ in a minimization problem,
\[
f(\lambda')\le \mathbf c^\top \mathbf x^{\star}=f(\hat\lambda).
\]
Similarly, if it is feasible for $\lambda''$, then
\[
f(\lambda'')\le \mathbf c^\top \mathbf x^{\star}=f(\hat\lambda).
\]
Therefore,
\[
f(\hat\lambda)\ge \min\{f(\lambda'),f(\lambda'')\},
\]
which proves quasi-concavity.
\end{proof}
\noindent
\textbf{Observation:} the proof extends to the mixed-integer case since it relies on the same solution $\mathbf{x}^{\star}$ and does not construct new solutions through convex combinations. Hence, any integrality constraints satisfied by $\mathbf{x}^{\star}$ are preserved.

\begin{corollary}
Suppose $D$ has exactly one non-zero row. Then, the minimum value of $f$ over $[\lambda_1,\lambda_2]$ is attained at one of the endpoints, namely
\[
\min_{\lambda\in[\lambda_1,\lambda_2]} f(\lambda)
=
\min\{f(\lambda_1),f(\lambda_2)\}.
\]
\end{corollary}
\begin{corollary}
Suppose $D$ has exactly one non-zero row.
Then, a maximum of $f$ over $[\lambda_1,\lambda_2]$ can be approximated within accuracy $\varepsilon$ by applying the standard bisection method for quasiconvex optimization to $-f$ (see \cite{boyd2004convex}).
In particular, given an initial interval $[l,u]$ containing the optimal value of $\min_{\lambda\in[\lambda_1,\lambda_2]} -f(\lambda)$, the method terminates in
\begin{equation*}
O\!\left(
\log\!\left(
\frac{u-l}{\varepsilon}
\right)
\right)
\end{equation*}
iterations, each iteration requiring to solve one LP.
\end{corollary}

\noindent
\textbf{Observation:} In this case again, having exactly one non-zero row in $D$ cannot occur for models with parametric equality constraints, since such an equality is written as two inequalities to match the canonical model considered in this paper.

\subsubsection{Exactly one non-zero column in \texorpdfstring{$D$}{D}}
We consider the case where $D$ has exactly one non-zero column.
This implies that only one variable, denoted by $x_0$, is affected by the parameter $\lambda$.
We decompose the variable vector as $(x_0, \mathbf x)$, where $\mathbf x \in \mathbb{R}^{n-1}$.
Let $A_0$ denote the corresponding column of $A$ and $\mathbf d$ the non-zero column of $D$.
The problem can then be written as
\begin{equation*}
\begin{aligned}
    \min_{\mathbf{x}, \, x_0, \, \lambda} \quad & \mathbf{c}^\top \mathbf{x} + c_0 x_0 \\
    \text{s.t.} \quad & A\mathbf{x} + x_0(A_0 + \lambda \mathbf{d}) \le \mathbf{b}, \\
    & \mathbf{x} \ge \mathbf{0}, \quad x_0 \ge 0, \\
    & \lambda \in [\lambda_1, \lambda_2].
\end{aligned}
\tag{$\mathcal{P}_{\min,\text{col}}$}\label{prob:p_min_col}
\end{equation*}
In this setting, the value function $f(\lambda)$ is quasi-convex on $[\lambda_1, \lambda_2]$, 
which implies that the problems~\eqref{min--min} and~\eqref{max--min} are no longer NP-hard.
\begin{definition}
A function $f : \mathbb{R}^n \to \mathbb{R}$ is called \emph{quasi-convex}
if its domain $\mathrm{dom}(f)$ and all its sublevel sets
\[
S_t = \{x \in \mathrm{dom}(f) \mid f(x) \le t\},
\qquad \forall t \in \mathbb{R},
\]
are convex.
\end{definition}
\begin{proposition}\label{prop:one_col}
If $D$ has exactly one non-zero column, then the value function $f(\lambda)$ is quasi-convex on $[\lambda_1, \lambda_2]$.
\end{proposition}

\begin{proof}

Let
\[
S_t = \{\lambda \in [\lambda_1,\lambda_2] \mid f(\lambda)\le t\}.
\]
We show that $S_t$ is convex.

Take $\lambda',\lambda'' \in S_t$.
Then, $ f(\lambda')\le t \quad \text{and} \quad f(\lambda'')\le t $.
Hence, there exist feasible solutions $(\mathbf x',x_0')$ for $\lambda'$ and $(\mathbf x'',x_0'')$ for $\lambda''$,
such that
\[
\mathbf c^\top \mathbf x' + c_0 x_0' \le t,
\qquad
\mathbf c^\top \mathbf x'' + c_0 x_0'' \le t,
\]
and
\[
A\mathbf x' + (A_0+\lambda' \mathbf d)x_0' \le \mathbf b,
\qquad
A\mathbf x'' + (A_0+\lambda'' \mathbf d)x_0'' \le \mathbf b.
\]

Let $\alpha \in [0,1]$ and let $\hat\lambda = \alpha \lambda' + (1-\alpha)\lambda''$.
To prove that $\hat\lambda\in S_t$, we must find a feasible solution for parameter $\hat\lambda$ with objective value at most $t$.

The idea is to construct such a solution from the feasible solutions associated with \(\lambda'\) and \(\lambda''\).
For this purpose, we introduce the coefficient

\[
\theta
=
\frac{\alpha x_0''}
{\alpha x_0'' + (1-\alpha)x_0'},
\quad \text{and \ therefore} \quad
1-\theta
=
\frac{(1-\alpha)x_0'}
{\alpha x_0'' + (1-\alpha)x_0'}.
\]
Set
\[
\hat{\mathbf x}
=
\theta \mathbf x'
+
(1-\theta)\mathbf x''
\quad \text{and} \quad
\hat x_0
=
\theta x_0'
+
(1-\theta)x_0''.
\]
By construction,
\[
\hat x_0
=
\frac{x_0' x_0''}
{\alpha x_0'' + (1-\alpha)x_0'}.
\]
Moreover,
\[
\hat\lambda \hat x_0
=
(\alpha\lambda' + (1-\alpha)\lambda'')
\frac{x_0'x_0''}
{\alpha x_0'' + (1-\alpha)x_0'}.
\]
On the other hand,
\[
\theta \lambda' x_0'
+
(1-\theta)\lambda'' x_0''
=
\frac{\alpha\lambda' x_0' x_0''
+
(1-\alpha)\lambda'' x_0' x_0''}
{\alpha x_0'' + (1-\alpha)x_0'}.
\]
Thus,
\[
\hat\lambda \hat x_0
=
\theta \lambda' x_0'
+
(1-\theta)\lambda'' x_0''.
\]
This identity allows the nonlinear term $\lambda x_0$ to be expressed as a convex combination of the corresponding terms associated with $\lambda'$ and $\lambda''$.

\noindent
We now verify feasibility, i.e.,
\[
A\hat{\mathbf x}
+
(A_0+\hat\lambda \mathbf d)\hat x_0.
\]
Using the definitions above,
\[
\begin{aligned}
A\hat{\mathbf x}
+
(A_0+\hat\lambda \mathbf d)\hat x_0
&=
A(\theta\mathbf x' + (1-\theta)\mathbf x'')
\\
&\quad
+
A_0(\theta x_0' + (1-\theta)x_0'')
\\
&\quad
+
\mathbf d\, \hat\lambda \hat x_0.
\end{aligned}
\]
Since
\[
\hat\lambda \hat x_0
=
\theta \lambda' x_0'
+
(1-\theta)\lambda'' x_0'',
\]
we obtain
\[
\begin{aligned}
A\hat{\mathbf x}
+
(A_0+\hat\lambda \mathbf d)\hat x_0
&=
\theta
\left(
A\mathbf x'
+
(A_0+\lambda'\mathbf d)x_0'
\right)
\\
&\quad
+
(1-\theta)
\left(
A\mathbf x''
+
(A_0+\lambda''\mathbf d)x_0''
\right).
\end{aligned}
\]
Because both vectors inside the parentheses satisfy the constraints
\[
A\mathbf x' + (A_0+\lambda'\mathbf d)x_0' \le \mathbf b,
\quad \text{and} \quad
A\mathbf x'' + (A_0+\lambda''\mathbf d)x_0'' \le \mathbf b,
\]
their convex combination also satisfies them. Therefore,
\[
A\hat{\mathbf x}
+
(A_0+\hat\lambda \mathbf d)\hat x_0
\le \mathbf b.
\]
Hence, $(\hat{\mathbf x},\hat x_0)$ is feasible for $\hat\lambda$.

\noindent
Finally,
\[
\begin{aligned}
\mathbf c^\top \hat{\mathbf x}
+
c_0 \hat x_0
&=
\mathbf c^\top
(\theta\mathbf x' + (1-\theta)\mathbf x'')
+
c_0(\theta x_0' + (1-\theta)x_0'')
\\
&=
\theta
(\mathbf c^\top\mathbf x' + c_0x_0')
+
(1-\theta)
(\mathbf c^\top\mathbf x'' + c_0x_0'')
\\
&\le
\theta t + (1-\theta)t
\\
&= t.
\end{aligned}
\]
Thus, $(\hat{\mathbf x}, \hat x_0)$ is a feasible solution for $\hat{\lambda}$ whose objective does not exceed $t$.
Consequently,
\[
f(\hat\lambda)\le t,
\]
which implies that $\hat\lambda \in S_t$.

Since every convex combination of two points in $S_t$ remains in $S_t$, the set $S_t$ is convex.
Therefore every sublevel set of $f$ is convex, and hence $f$ is quasi-convex.
\medskip

\noindent
\textit{Remark.}
If $x_0'=x_0''=0$, the definition of $\theta$ is not valid.
However, in this case the bilinear term disappears and feasibility follows directly by convexity of the remaining linear constraints.
\end{proof}

\noindent
\textbf{Observation:} The proof does not extend to the mixed-integer case.
The constructed point $(\hat{\mathbf x},\hat x_0)$ is obtained as a convex combination
of two feasible solutions. Hence, even if the initial solutions are integer,
there is no guarantee that $(\hat{\mathbf x},\hat x_0)$ is integer as well.
This observation can be illustrated on the following instance:
\[
\begin{aligned}
\min_{x_1,x_2,\lambda}\quad & x_2\\
\text{s.t.}\quad
& x_1 \le 4,\\
& x_1 \ge 0,\quad x_2 \ge 0,\\
& x_1+x_2 \ge 1,\\
& (-2.05+\lambda)x_1 -4x_2 \le -2,\\
& (1.95-\lambda)x_1 -4x_2 \le 2,\\
& x_1 \in \mathbb Z,\quad x_2\in\mathbb R_+,\\
& \lambda\in[-0.5,2].
\end{aligned}
\]
Here, only one column of $D$ is non-zero, namely the column associated with
the integer variable $x_1$.
The continuous relaxation satisfies the quasi-convexity property proved
above. However, in the mixed-integer case, this property may fail. Indeed,
the value function $f(\lambda)$ is not quasi-convex, as illustrated in
Figure~\ref{fig:one_col_mip_counterexample}.
\begin{figure}[ht]
\centering
\includegraphics[width=0.9\textwidth]{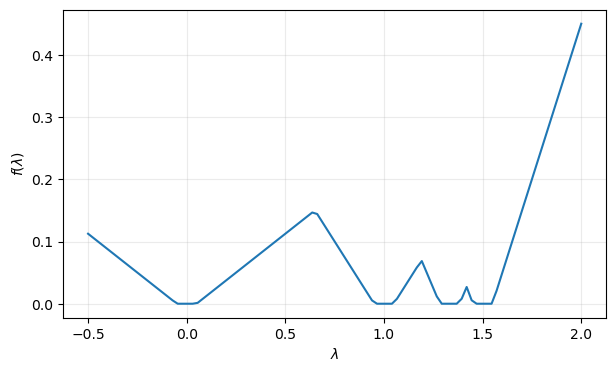}
\caption{Evolution of the value function $f(\lambda)$ for the mixed-integer instance with one non-zero column in $D$.}\label{fig:one_col_mip_counterexample}
\end{figure}

\begin{corollary}
Suppose $D$ has exactly one non-zero column.
Then, the maximum value of $f$ over $[\lambda_1,\lambda_2]$ is attained at one of the endpoints, namely
\[
\max_{\lambda\in[\lambda_1,\lambda_2]} f(\lambda)
=
\max\{f(\lambda_1),f(\lambda_2)\}.
\]
\end{corollary}
\begin{corollary}
Suppose $D$ has exactly one non-zero column.
Then, a minimum of $f$ over $[\lambda_1,\lambda_2]$ can be approximated
within accuracy $\varepsilon$ by applying the standard bisection method
for quasiconvex optimization (see \cite{boyd2004convex}).
In particular, given an initial interval $[l,u]$ containing the optimal value
of $\min_{\lambda\in[\lambda_1,\lambda_2]} f(\lambda)$, the method terminates in
\begin{equation*}
O\!\left(
\log\!\left(
\frac{u-l}{\varepsilon}
\right)
\right)
\end{equation*}
iterations, each iteration involving just the computation of one LP.
\end{corollary}

\noindent
\textbf{Observation:} Having exactly one non-zero column in $D$ can occur in models with parametric equality constraints, since the two inequalities encoding such an equality have their non-zero coefficients in the same column of $D$.

\section{Single-level Reformulations}\label{sec:reformulations}
The bilevel formulations \eqref{min--min} and \eqref{max--min} are not directly usable by standard optimization solvers.
We therefore rely on the reformulations introduced in Section~\ref{subsec:aux} to derive equivalent single-level problems that can be handled by nonlinear optimization solvers.

The min--min problem~\eqref{min--min} requires no further work.
As shown in Section~\ref{subsec:aux}, it is already equivalent to the single-level problem~\ref{prob:p_min}.

We then consider the max--min problem~\eqref{max--min}.
Unlike the min--min case, the formulation remains bilevel and cannot be directly reduced to a single-level optimization problem.
Our reformulation relies on the lower-level problem being feasible for all values of $\lambda$ in the interval and having a finite optimal value.
To verify feasibility, we use Farkas' lemma \cite{Williamson2008Lecture7}, which provides a certificate of infeasibility.
Applying the lemma to the system
\begin{equation*}
(A+\lambda D)\mathbf x \le \mathbf b,\qquad \mathbf x \ge \mathbf 0,    
\end{equation*}
we obtain that infeasibility occurs for some $\lambda$ if and only if there exist $\lambda$ and $\mathbf y$ such that
\begin{equation*}
(A+\lambda D)^\top \mathbf y \ge 0,\qquad \mathbf y^\top \mathbf b = -1,\qquad \mathbf y \ge \mathbf 0.    
\end{equation*}
The lower-level problem is therefore feasible for all $\lambda \in [\lambda_1,\lambda_2]$ if and only if the feasibility problem
\begin{equation}
\begin{aligned}
\max \quad & 0 \\
\text{s.t.} \quad
& (A+\lambda D)^\top \mathbf y \ge \mathbf 0, \\
& \mathbf y^\top \mathbf b = -1, \\
& \mathbf y \ge \mathbf 0, \\
& \lambda_1 \le \lambda \le \lambda_2
\end{aligned}
\label{feasprob}
\tag{$\mathcal{F}$}
\end{equation}
is infeasible.
Indeed, when~\eqref{feasprob} is infeasible, no certificate of infeasibility exists and the lower-level problem is feasible for every $\lambda \in [\lambda_1,\lambda_2]$.

Assume furthermore that the lower-level problem has a finite optimal value for every $\lambda$.
Strong duality then allows the lower level to be replaced by its dual representation, which yields the single-level formulation
\begin{equation}
\begin{aligned}
\max_{\lambda,\mathbf y}\quad
&\mathbf b^\top\mathbf y\\
\text{s.t.}\quad
&(A+\lambda D)^\top\mathbf y\le\mathbf c,\\
&\mathbf y\le0,\\
&\lambda_1\le\lambda\le\lambda_2.
\end{aligned}
\label{max_reformulation}
\tag{$\mathcal{P}_{\max}$}
\end{equation}
Under these two assumptions, namely that~\eqref{feasprob} is infeasible and that the lower-level problem is bounded for every $\lambda\in[\lambda_1,\lambda_2]$, the max--min problem is equivalent to~\eqref{max_reformulation}.

\section{Experiments}\label{sec:experiments}
This section reports a computational study of the three single-level models derived in
Section~\ref{sec:reformulations}.
These are the min--min formulation~\eqref{prob:p_min}, the Farkas feasibility
problem~\eqref{feasprob} and the max--min reformulation~\eqref{max_reformulation}.
We report the joint solver outcomes of the three models and their running times.
We compare these running times with the lower-level problem solved at a fixed value of $\lambda$.
We then show that parametric equality constraints make~\eqref{max_reformulation} hard in practice.
Finally, we study the effect of explicit bounds on the dual variables.

\subsection{Benchmark and setup}\label{sec:exp-benchmark}
We use the parametric linear instances proposed by Derval et al.~\cite{miftari2024sensitivity},
which are derived from the Netlib collection.
The constraints of the problems are of the forms
\begin{equation*}
A\mathbf x=\mathbf b,\qquad A\mathbf x\le \mathbf b,\qquad (A+\lambda D)\mathbf x=\mathbf b,\qquad (A+\lambda D)\mathbf x\le \mathbf b,
\end{equation*}
with 100 constraints depending on the parameter~$\lambda$.
These 100 constraints are split between equalities and inequalities, the number of each type
ranging from 0 to 100 across the collection.
For each instance, we consider the same interval for the parameter $\lambda$, namely $\Lambda=[-1,1]$.
Table~\ref{tab:instance-sizes} summarizes the size distribution.


\begin{table}[ht]
\centering
\begin{tabular}{lrrrr}
\hline
Statistics & Min & Median & Mean & Max \\
\hline
Vars. & 79 & 1\,175 & 2\,028 & 13\,525 \\
Total Constr. & 153 & 1\,966 & 3\,240 & 24\,025 \\
$Ax=b$ & 0 & 287 & 495 & 3,115 \\
$Ax\leq b$ & 14 & 1\,486 & 2\,644 & 21\,025 \\
$(A+\lambda D)x=b$ & 0 & 18 & 35 & 100 \\
$(A+\lambda D)x\leq b$ & 0 & 82 & 65 & 100 \\
\hline
\end{tabular}
\caption{Size statistics for the 191 benchmark instances.}
\label{tab:instance-sizes}
\end{table}

On eleven of the 191 instances, the min--min model~\eqref{prob:p_min} is infeasible, so the
lower-level problem admits no feasible point at any $\lambda\in\Lambda$ and the max--min value is
$+\infty$ by Remark~\ref{remark:bilevel_inf_unb}. All numbers below refer to the remaining 180
instances. On those, the nominal problem ($\lambda=0$) is solved to optimality in a median of $0.005$ seconds, a mean of $0.020$ seconds and a
maximum of $0.154$ seconds. A single linear program at a fixed value of the parameter is thus
computationally negligible and the running times we report below are the cost of
optimizing over the interval $\Lambda$.

All models are implemented in Python and solved with Gurobi Optimizer 12.0.3 on a machine running
Linux Mint 22.3 with an AMD Ryzen 7 5825U processor (8 physical cores, 16 logical processors) and
35~GiB of usable RAM. Gurobi is allowed to use all 16 threads, and each run is given a time limit
of 1\,200 seconds.

Bilinear terms are passed to the solver through auxiliary variables. We introduce
$z_i=\lambda x_i$ for the min--min model and $w_i=\lambda y_i$ for~\eqref{feasprob}
and~\eqref{max_reformulation}, add the defining bilinear equalities, and let Gurobi handle the
resulting non-convex quadratic constraints through its spatial branch-and-bound. No relaxation of
these products is performed on our side.

\subsection{Results}\label{sec:exp-outcomes}
For each instance we solve the min--min model~\eqref{prob:p_min}, then~\eqref{feasprob}, and then the max--min
model~\eqref{max_reformulation} unless~\eqref{feasprob} returned a certificate.
When~\eqref{feasprob} reaches the time limit, lower-level feasibility over $\Lambda$ is
undecided but we still solve~\eqref{max_reformulation}.

\begin{table}[ht]
\begin{tabular}{lllrr}
\hline
Minimum & Lower-level feasibility & Maximum & \#instances & \%\\
\hline
optimal & feasible for all $\lambda$ & optimal & 83 & 46.1\\
optimal & time limit & time limit & 42 & 23.3\\
optimal & feasible for all $\lambda$ & time limit & 34 & 18.9\\
optimal & time limit & unbounded & 12 & 6.7\\
optimal & feasible for all $\lambda$ & suboptimal & 3 & 1.7\\
optimal & infeasible for some $\lambda$ & skipped & 3 & 1.7\\
optimal & feasible for all $\lambda$ & unbounded & 2 & 1.1\\
suboptimal & feasible for all $\lambda$ & time limit & 1 & 0.6\\
\hline
\end{tabular}
\caption{Joint solver outcomes for the min--min formulation, the feasibility check~\eqref{feasprob}, and the max--min formulation over the 180 retained
instances.}
\label{tab:combinations}%
\end{table}

Table~\ref{tab:combinations} summarizes the joint outcomes.
The min--min problem reaches a proven finite optimum on every instance but
one, which terminates with an incumbent the solver could not certify within tolerances.
The feasibility check is decided on 70\% of the instances and reaches
the time limit on the rest.
The max--min reformulation is the bottleneck.
It reaches a proven finite optimum on only $46.1\%$ of them.

On three instances~\eqref{feasprob} is feasible and returns a Farkas certificate that the
lower-level problem is infeasible at some $\lambda\in\Lambda$, namely at $\lambda=-1$,
$\lambda=0.651$ and $\lambda=0.096$. The max--min value is $+\infty$ on these three,
and~\eqref{max_reformulation} is not solved.
Figure~\ref{fig:holes} shows their value function over $\Lambda$.
On each of them the lower level is infeasible on a whole subinterval rather than at an isolated
point, and is feasible at $\lambda=0$.

\begin{figure}[ht]
\centering
\includegraphics[width=\textwidth]{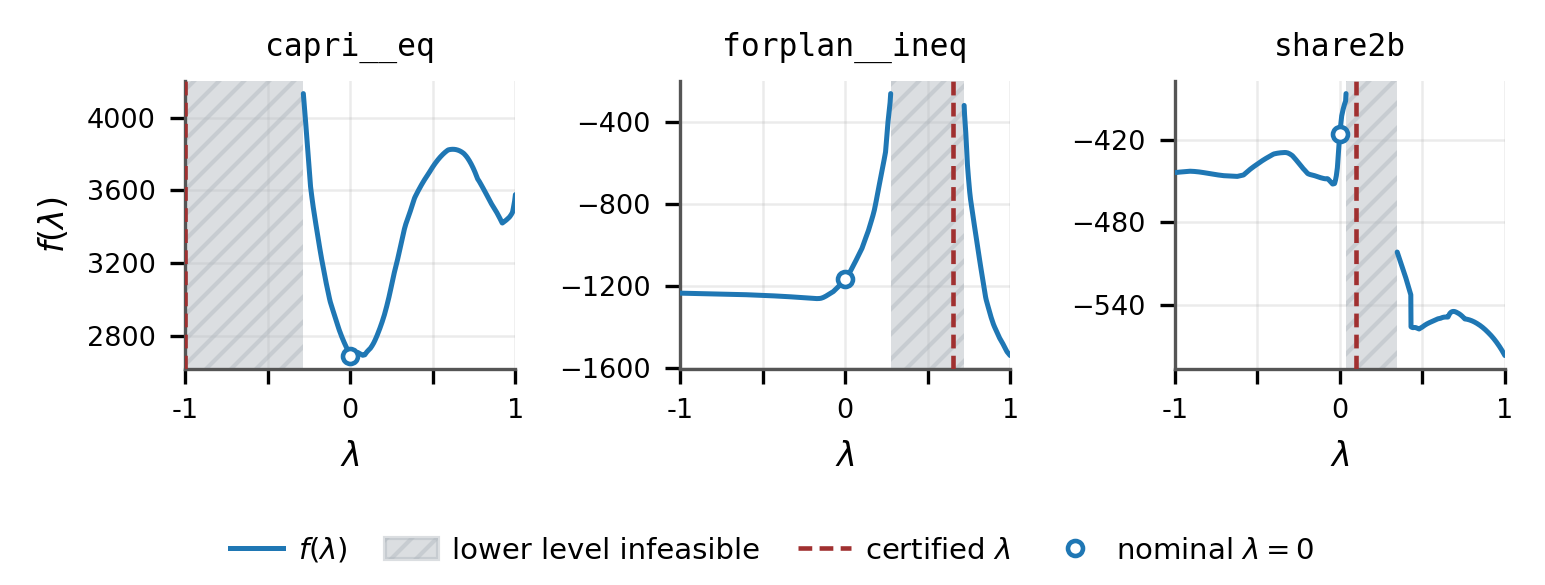}
\caption{Value function $f(\lambda)$ over $\Lambda=[-1,1]$ on the three instances for
which~\eqref{feasprob} returns a certificate. The hatched band is the set of $\lambda$ at which the
lower-level problem is infeasible, obtained by solving it at $1\,601$ equally spaced values of
$\lambda$. The dashed line is the value certified by~\eqref{feasprob}.}\label{fig:holes}
\end{figure}

On two instances, \eqref{feasprob} certifies that the lower-level problem is feasible for every $\lambda\in\Lambda$ while \eqref{max_reformulation} ends with an unbounded status.
Under our assumptions, the max--min value must be finite, making the unbounded status returned by Gurobi 12.0.3 for \eqref{max_reformulation} inconsistent with this certificate.
On both instances, Gurobi finds an incumbent whose objective value equals the maximum of $f$ over $\Lambda$, but does not certify its optimality.
Further investigation shows that the solver BARON also reports an unbounded status, whereas Gurobi 13 certifies a finite optimal solution on both instances.
Both instances contain parametric equality constraints, whose free dual variables enter bilinear terms in \eqref{max_reformulation}. These observations suggest difficulties in handling the resulting nonconvex formulation.

\begin{table}[ht]
\small
\begin{tabular}{llrrrr}
\hline
Model & Status & \#runs & Median & Mean & Max\\
\hline
Nominal problem ($\lambda=0$) & optimal & 180 & 0.005 & 0.020 & 0.154\\
\hline
min--min~\eqref{prob:p_min} & optimal & 179 & 0.23 & 2.11 & 53.0\\
 & suboptimal & 1 & 0.27 & 0.27 & 0.27\\
\hline
\eqref{feasprob} & feasible for all $\lambda$ & 123 & 0.06 & 10.7 & 667.0\\
 & certificate found & 3 & 0.25 & 0.84 & 2.20\\
 & time limit & 54 & 1200 & 1200 & 1200\\
\hline
max--min~\eqref{max_reformulation} & optimal & 83 & 0.33 & 3.97 & 203.0\\
 & suboptimal & 3 & 98.9 & 109.9 & 230.5\\
 & unbounded & 14 & 349.0 & 442.4 & 1185.6\\
 & time limit & 77 & 1200 & 1200 & 1200\\
 & skipped & 3 & -- & -- & --\\
\hline
\end{tabular}
\caption{Running times in seconds, by model and termination status, over the 180 retained instances.}
\label{tab:runtimes}%
\end{table}

Table~\ref{tab:runtimes} reports the running-time distribution of each model.
When the exact formulations terminate, they are fast.
The median solved run of~\eqref{max_reformulation} takes $0.33$ seconds, slightly more than the
min--min model.
A run therefore either terminates quickly or reaches the time limit, with little in between.
The optimality gaps show the same thing: of
the 77 runs that reach the time limit, 67 still have an infinite gap, the solver having
produced no finite dual bound at all. The feasibility check is not a cheap side computation
either, with 54 time limits and a worst terminating run of 667 seconds.

\subsection{Why the max--min reformulation is hard}\label{sec:exp-why}
Both single-level models contain bilinear terms $\lambda x_i$ or $\lambda y_i$, which Gurobi relaxes by McCormick envelopes~\cite{mccormick1976,gurobiminlp}.
In~\eqref{prob:p_min}, $x_i\ge0$ is bounded. In~\eqref{max_reformulation}, the dual variable $y_i$ is free when the
primal has equality constraints, so the relaxation can be unbounded and the solver might find no finite dual bound.
This is the case in the 67 runs that end with an infinite gap.
All except one of the 77 time limits of~\eqref{max_reformulation} occur on an instance with at least one parametric equality (Table~\ref{tab:eqsplit}).

The magnitude of the dual variables is a candidate explanation for the practical difference
between the two models. We explore it by re-solving the 77 time-limit instances after adding a
box $-M\le y_i\le M$ on every dual variable. We take $M=10^{4}$ and a reduced time limit of 120
seconds. On the 35 instances on which~\eqref{feasprob} proved the lower-level problem feasible
for all $\lambda$, the box is decisive. Every one of them reaches a proven optimum, in a median
of $0.73$ seconds, where 1\,200 seconds had left most of them without even a finite dual bound.

\begin{table}[ht]
\small
\begin{tabular}{lrrrrrr}
\hline
& \multicolumn{2}{c}{min--min~\eqref{prob:p_min}} & \multicolumn{2}{c}{\eqref{feasprob}} & \multicolumn{2}{c}{max--min~\eqref{max_reformulation}}\\
\cmidrule(lr){2-3}\cmidrule(lr){4-5}\cmidrule(lr){6-7}
Parametric equalities & solved & time limit & solved & time limit & solved & time limit\\
\midrule
none (63) & 63 & 0 & 63 & 0 & 58 & 1\\
at least one (117) & 116 & 0 & 63 & 54 & 25 & 76\\
\hline
\end{tabular}
\caption{Solver outcomes split by whether the instance has at least one parametric equality constraint $(A+\lambda D)x=b$,
over the 180 retained instances. Only two of the possible outcomes are reported.}
\label{tab:eqsplit}
\end{table}

\begin{table}[ht]
\small
\begin{tabular}{lrrrr}
\hline
& \multicolumn{2}{c}{feasible for all $\lambda$ (35)} & \multicolumn{2}{c}{time limit (42)}\\
\cmidrule(lr){2-3}\cmidrule(lr){4-5}
& no bounds & $|y_i|\le10^4$ & no bounds & $|y_i|\le10^4$\\
& 1\,200\,s & 120\,s & 1\,200\,s & 120\,s\\
\hline
Proven optimal & 0 & 35 & 0 & 34\\
Finite dual bound obtained & 10 & 35 & 0 & 42\\
Infinite gap at termination & 25 & 0 & 42 & 0\\
Median running time (s) & 1200 & 0.73 & 1200 & 18.0\\
Box active at the returned solution & --- & 24/35 & --- & 36/42\\
\hline
\end{tabular}
\caption{Effect of an explicit box $|y_i|\le10^{4}$ on the dual variables, on the 77 instances
for which~\eqref{max_reformulation} reaches the 1\,200-second time limit without bounds. The
boxed model is given only 120 seconds. The two groups differ in whether~\eqref{feasprob} proved
the lower-level problem feasible for all $\lambda$ within its own time
limit.}
\label{tab:dualbounds}
\end{table}

The box, however, modifies the problem being solved.
By linear programming duality, bounding $|y_i|\le M$ on an equality row adds a slack variable to
that row in the lower-level problem. The slack is priced at $M$ per unit of violation.
The boxed model therefore solves a relaxation in which every equality may be violated at a large
but finite cost. A dual variable sitting at its bound means that this slack is used. That
equality is then expensive to satisfy exactly. The boxed value is only a lower bound on the
max--min value. This happens on 24 of these 35 instances. On the 11 others the box is inactive at
both $M=10^{4}$ and $M=10^{5}$. The returned value is then the max--min value itself.

\section{Conclusion}\label{sec:conclusions}
In this paper, we consider parametric linear programs with constraint matrices depending linearly on a scalar parameter.
We study the problem of computing the minimum and maximum of the associated value function over a bounded interval.
We prove that both corresponding decision versions of the bilevel problems are NP-complete, and we identify restricted assumptions under which the bilevel programs are no longer NP-hard.



The min--min problem reduces directly to a single-level bilinear program, whereas the max--min problem requires a strong-duality reformulation that is valid only when the lower level is feasible over the whole parameter interval.
Rather than assuming this condition, as is common in the bilevel literature, we certify it through a Farkas-based feasibility problem.

The computational study on 180 instances derived from Netlib reveals a clear practical asymmetry between the two problems.
The min--min formulation is solved on all but one instance, with a median running time below one second.
In contrast, the max--min reformulation reaches a proven optimum on less than half of the instances, and the feasibility check itself reaches the time limit on 54 of them.
The difficult instances are almost exclusively those carrying a parametric equality constraint.
Their dual variables are free, the McCormick relaxation of the corresponding bilinear terms is then unbounded, and the solver terminates without any finite dual bound.
Imposing bounds on these variables restores a finite relaxation and lets the solver terminate, but the bounded model is no longer exact.

Several questions remain open.
Bounds on the dual variables that are valid rather than imposed would keep the reformulation exact and make its relaxation informative.
A second direction is algorithmic. The problem depends on a single scalar parameter, and a dedicated method could branch on this parameter alone.


\newpage

\section*{Statements and Declarations}

\begin{itemize}
\item \textbf{Competing Interests.} The authors declare that they have no conflict of interest.
\item \textbf{Data Availability.} The benchmark instances used in this study come from Derval et al.~\cite{miftari2024sensitivity} and are publicly available on Zenodo at \url{https://zenodo.org/records/18804576}.
\item \textbf{Funding.} No funding was received for conducting this study.
\end{itemize}

\begin{appendices}

\section{Auxiliary results used in the proof of Theorem~\ref{npinclusion}}\label{appendix:np}

\begin{lemma}
\label{lem:parametric-det}
Let $M(\lambda)$ be an $m\times m$ matrix whose entries are affine functions
of $\lambda$ with rational coefficients. Then, the following holds:
\begin{enumerate}[label=(\roman*)]
\item\label{lemma:assertion1} $\det(M(\lambda))$ is a polynomial of degree at most $m$;
\item\label{lemma:assertion3} the coefficients of $\det(M(\lambda))$ can be computed in polynomial time in the encoding length of $M$;
\item\label{lemma:assertion2} the coefficients of $\det(M(\lambda))$ have encoding length polynomial in the encoding length of $M$.
\end{enumerate}
\end{lemma}

\begin{proof}

By the Leibniz formula, we have
\begin{equation*}
\det(M(\lambda))
=
\sum_{\sigma\in S_m}
\operatorname{sgn}(\sigma)
\prod_{i=1}^m
M_{i,\sigma(i)}(\lambda),
\end{equation*}
where $S_m$ is the permutation group. Each factor $M_{i,\sigma(i)}(\lambda)$ is affine in $\lambda$, hence every product has degree at most
$m$, proving assertion~\ref{lemma:assertion1}.

Next, we prove assertions~\ref{lemma:assertion3} and~\ref{lemma:assertion2}. By assertion~\ref{lemma:assertion1}, the polynomial $\det(M(\lambda))$ has degree at most $m$. Hence, it is uniquely determined by its values at any $m+1$ distinct points. For each integer $\lambda^\prime \in\{0,1,\ldots,m\}$, the matrix $M(\lambda^\prime)$ can be constructed in polynomial time and has integer entries of polynomial encoding length. Therefore, each determinant $\det(M(\lambda^\prime))$ can be computed in polynomial time using any polynomial-time algorithm for the determinant of an integer matrix.

Having computed the values $\det(M(0)), \det(M(1)),\ldots, \det(M(m))$, the coefficients of the unique polynomial $\det(M(\lambda))$ of degree at most $m$ are recovered by polynomial interpolation on these $m+1$ evaluation points, which can be performed in polynomial time using standard polynomial interpolation algorithms. Hence, these coefficients have encoding length polynomial in the encoding length of M.
\end{proof}

\begin{lemma}
\label{lem:rational-to-integer}
Let
\begin{equation*}
f_1(\lambda),\ldots,f_s(\lambda)
\in
\mathbb{Q}[\lambda]
\end{equation*}
be univariate polynomials.
Then there exist other univariate polynomials
\begin{equation*}
g_1(\lambda),\ldots,g_s(\lambda)
\in
\mathbb{Z}[\lambda] 
\end{equation*}
that can be computed in polynomial time such that, for every $i=1, \ldots, s$,
\begin{equation*}
f_i(\lambda)\;\Delta_i\;0 \quad\Longleftrightarrow\quad g_i(\lambda)\;\Delta_i\;0,
\qquad
\Delta_i\in
\{<,\le,=,\ge,>\},
\end{equation*}
\end{lemma}

\begin{proof}
 Fix a polynomial
\begin{equation*}
    f(\lambda)=\sum_{k=0}^{d}\frac{a_k}{b_k}\lambda^k 
\end{equation*}
 where $a_k\in\mathbb{Z}$ and $b_k\in\mathbb{Z}^+_0$.
 Let 
 \begin{equation*}
    D=\prod_{k=0}^{d} b_k     
 \end{equation*}
 and define
 \begin{equation*}
     g(\lambda)=D\,f(\lambda)\in\mathbb{Z}[\lambda].
 \end{equation*}
 Since $D>0$, multiplication by $D$ preserves all comparison relations with zero.
 Moreover, the binary encoding length of $D$ is the sum of the encoding lengths of the denominators $b_k$, and therefore is polynomial in the encoding length of $f$. Consequently, the coefficients of $g$ also have polynomial encoding length.
\end{proof}

\begin{lemma}[Renegar {\cite[Theorem~1.1]{renegar1992}}]
\label{lem:univariate}

Let
\begin{equation*}
f_1(\lambda),\ldots,f_s(\lambda)
\in
\mathbb{Z}[\lambda]
\end{equation*}
be univariate polynomials of degree at most $d$, whose coefficients have
binary encoding length at most $L$. Then, the feasibility of a system of polynomial constraints
\begin{equation*}
f_i(\lambda)\;\Delta_i\;0,
\qquad
\Delta_i\in
\{<,\le,=,\ge,>\},
\end{equation*}

can be decided in time polynomial in
$s$, $d$ and $L$.
\end{lemma}

\begin{proof}
The result is an immediate specialization of Theorem~1.1 of Renegar~\cite{renegar1992}
to the case of a single real variable.
\end{proof}

\end{appendices}
\bibliographystyle{plain}
\bibliography{ref}
\end{document}